\documentclass{amsart}
\usepackage{amssymb}
\usepackage{epsfig}

 \newcommand{\Z}{{\mathbb Z}}
 \newcommand{\C}{{\mathbb C}}
 \newcommand{\A}{{\mathbb A}}
\newcommand{\Q}{{\mathbb Q}}
 
 \newcommand{\N}{{\mathbb N}}
\newcommand{\T}{{\mathbb T}}
\newcommand{\cH}{\mathcal H}
\newcommand{\p}{\partial}

\newcommand{\ve}{\varepsilon}
 \newtheorem{theorem}{Theorem}
 \newtheorem{lemma}[theorem]{Lemma}
  \newtheorem{proposition}[theorem]{Proposition}
 \newtheorem{corollary}[theorem]{Corollary}
 \newtheorem{remark}[theorem]{Remark}
 \newtheorem{example}[theorem]{Example}
 
 \newtheorem{conjecture}[theorem]{Conjecture}
\newcommand{\g}{\mathfrak g}
\newcommand{\h}{\mathfrak h}

\def\Box
  {\hfill \thinspace\vbox{\hrule height .5pt \hbox{\vrule
   width .5pt \vbox to 7pt{\hbox to 3.5pt{}} \vrule width .5pt}
   \hrule height 0pt depth .5pt}}

\title{Quantizations of Character Varieties and Quantum Knot Invariants}
\author{Adam S. Sikora}


\date{}

\begin{document}
\thispagestyle{empty}

\begin{abstract}
Let $G$ be a simple complex algebraic group and $\g$ its Lie algebra.
We show that the $\g$-Witten-Reshetikhin-Turaev quantum invariants
determine a deformation-quantization, $\C_q[X_G(\text{torus})],$
of the coordinate ring of the $G$-character variety of the torus.
We prove that this deformation is in the direction of the Goldman's bracket.
Furthermore, we show that every knot $K\subset S^3$ defines an ideal $I_K$ in $\C_q[X_G(\text{torus})].$
We conjecture that the homomorphism $\ve: \C_q[X_G(\text{torus})]\to \C[X_G(\text{torus})],$
$q\to 1,$ maps $I_K$ to the ideal whose radical is the kernel of the map
$\C[X_G(\text{torus})]\to \C[X_G(S^3\setminus K)].$
This conjecture is related to AJ-conjecture for $sl(2,\C)$.
The results of this paper are inspired by the theory of q-holonomic relations between quantum
invariants of Garoufalidis and Le. Along the way, we disprove Conjecture 2 in \cite{LeTB}.
\end{abstract}

\maketitle
\pagestyle{myheadings}
\markboth{\hfil{\sc ADAM S. SIKORA}\hfil}
{\hfil{\sc QUANTIZATIONS OF CHARACTER VARIETIES}\hfil}

%
\section{Statements of theorems and conjectures}
%

\subsection{Quantum link invariants}
For a simple complex Lie algebra $\g$, denote by $WRT_{\g,V}(L)$ the $(\g,V)$-Witten-Reshetikhin-Turaev invariant of a framed oriented link $L\subset S^3$ whose all components are labeled by a finite dimensional representation $V,$ \cite{RT}.\footnote{$\g$-quantum invariants are also called by the
compact Lie group corresponding to the real compact form of $\g.$ For example $sl(n,\C)$-quantum
invariants are also called $SU(n)$-quantum invariants.}
It is a polynomial in  $q^{\pm \frac{1}{2D(\g)}},$
where $D(\g)$ is the smallest positive integer such that the dual Killing form, $(\cdot,\cdot):
\Lambda_\g^*\times \Lambda_\g^* \to \Q,$ has all its values in $\frac{1}{D(\g)}\Z$, c.f. \cite{LeQI}.
($\Lambda_\g$ denotes the weight lattice of $\g.$
The form is normalized so that $(\lambda,\lambda)=2$ for all short roots $\lambda$.)
\begin{equation}\label{Dg}
D(\g)=\begin{cases}
n & \text{for}\ \g=sl(n)\\
1 & \text{for}\ \g=sp(2n),so(4n+1), E_8,F_4,G_2\\
2 & \text{for}\ \g=so(4n),so(4n+3), E_7\\
3 & \text{for}\ \g=E_6,\\
4 & \text{for}\ \g=so(4n+2).\\
\end{cases}
\end{equation}

Given a Cartan subalgebra $\h\subset \g$ and fixed positive roots of $\g$, each finite dimensional
irreducible representation of $\g$ is determined by its highest weight $\lambda.$
We denote that representation by $V(\lambda).$
Then $\lambda \to WRT_{\g,V(\lambda)}(L)$ is a function defined on the
set of all dominant weights.

\subsection{Extension to Verma modules}
For the purpose of relating Witten-Reshetikhin-Turaev invariants of links to the topology of their
complements, we need to extend the above function to the entire weight lattice of $\g.$
However, if $\lambda$ is not dominant then all representations of $\g$ with highest
weight $\lambda$ are infinite dimensional. In fact, each $\lambda\in \Lambda_\g$ defines the Verma module
$M(\lambda)$ which is an infinite-dimensional indecomposable $\g$-module of highest weight $\lambda$
with the universal property that each indecomposable $\g$-module of highest weight $\lambda$ is a quotient of $M(\lambda).$
Rozanski (for $sl(2)$, \cite{Ro}) and Le (for all $\g$) made the following surprising observation:
Reshetikhin-Turaev construction of quantum invariants for knots (but not links) extends verbatim for all Verma modules of $\g$. (Details in Sec. \ref{s_WRT}.)
Furthermore,
\begin{equation}\label{M=V}
WRT_{\g,M(\lambda)}(K)=WRT_{\g,V(\lambda)}(K),
\end{equation}
for all dominant weights $\lambda.$
Let
\begin{equation}\label{WRT_J}
J_{\g,K}: \Lambda_\g\to \C[q^{\pm 1/D(\g)}],\quad J_{\g,K}(\lambda)=WRT_{\g,M(\lambda-\rho)}(K),
\end{equation}
where $\rho$ is the half-sum of positive roots of $\g.$ We call it
{\em the $\g$-Witten-Reshetikhin-Turaev function of $K$.} (The motivation for the shift by $\rho$
comes from Proposition \ref{J-sym}.)

\begin{example} For $\g=sl(2),$ $\rho=1\in \Lambda_\g=\Z.$
$J_{\g,K}(0)=0,$ $J_{\g,K}(1)=1,$ and $J_{\g,K}(2)$ for a zero-framed knot $K$ is the Jones polynomial of $K.$
More generally, $J_{\g,K}(n)$ is the Jones polynomial of $K$ colored by
the $n$-dimensional representation for $n\geq 1$ and
$J_{\g,K}(n)=-J_{\g,K}(-n)$ for negative $n.$
\end{example}

\begin{example}
It follows from \cite[1.4.4]{LeQI} that
\begin{equation}\label{unknot}
J_{\g,U}(\lambda)= \frac{\sum_{w\in W} sgn(w)q^{(\lambda,w(\rho))}}
{\sum_{w\in W} sgn(w)q^{(\rho,w(\rho))}},
\end{equation}
for the unknot $U$ and for every $\lambda\in \Lambda_\g.$ The sum is over the Weyl algebra of $\g.$
\end{example}

We are going to see that $J_{\g,K}$ has nice algebraic properties and it
encodes the $\g$-quantum invariants of $K$ in a form which is very convenient for the purpose of relating
them to the topology of $S^3\setminus K$.

\subsection{q-holonomicity}\label{ss_q-holon}
The next statement follows immediately from the argument of the proof of \cite[Thm 6]{GL1} --
see comments in Sec. \ref{s_WRT}.

\begin{theorem}\label{q-holon}
$J_{\g,K}$ a $q$-holonomic function on $\Lambda_g$ for all $\g\ne G_2.$
\end{theorem}

In order to define $q$-holonomicity of $J_{\g,K}$, consider the $\C[q^{\pm 1/D(\g)}]$-vector space
$F(\Lambda_g,\C[q^{\pm 1/D(\g)}])$ of all
$\C[q^{\pm 1/D(\g)}]$-valued functions on $\Lambda_\g$ and consider two families of operators on it:
$$E_\alpha f(\beta)=f(\alpha+\beta),\quad Q_\alpha f(\beta)=q^{(\alpha,\beta)}f(\beta),$$
for all $\alpha, \beta\in \Lambda_\g.$
Let $\A_g$ be the algebra of $\C[q^{\pm 1/D(\g)}]$-linear endomorphisms of $F(\Lambda_g,\C[q^{\pm 1/D(\g)}])$
generated by $E_\alpha$'s and $Q_\alpha$'s for $\alpha\in \Lambda_\g.$

\begin{proposition}\label{presentation}
$\A_\g$ is the $\C[q^{\pm 1/D(\g)}]$-algebra of polynomials in
non-commuting variables $E_\alpha,Q_\beta,$ $\alpha,\beta \in \Lambda_\g,$ subject to conditions:
\begin{equation}\label{qWeyl_rel}
E_\alpha E_\beta=E_{\alpha+\beta},\quad Q_\alpha Q_\beta=Q_{\alpha+\beta},\quad
E_\alpha Q_\beta=q^{(\alpha,\beta)}Q_\alpha E_\beta, \quad E_0=Q_0=1.
\end{equation}
\end{proposition}

We call $\A_\g$ the quantum Weyl algebra of $\g.$
$\A_{sl(2)}$ is the $q$-Weyl algebra of \cite{EO} and $q$-torus algebra
of \cite{GL1}:
$$\A_{sl(2)}=\C\langle E_1^{\pm 1},Q_1^{\pm 1}\rangle/E_1Q_1-q^{1/2}Q_1E_1.$$
(Recall that our dual Killing form is normalized so that $(\alpha,\alpha)=2$ for all short roots
$\alpha.$ For $sl(2,\C),$ $\alpha=\pm 2\in \Z=\Lambda_{\sl(2,\C)}.$ Hence $(1,1)=\frac{1}{2}.$
Therefore, our $q$ is $q^2$ in \cite{Ga1,Ga2,GL1}.)
For other $\g,$ $\A_\g$ appears to be different from the $q$-torus algebra
of \cite{GL1} and from other "quantum Weyl algebras" appearing for example in
\cite{DP, Gi, Ha, JZ, Ma, Pa, Ri}.

For any $f:\Lambda_\g\to \C[q^{\pm 1/D(\g)}]$ the set
$$I_f=\{P\in \A_\g: Pf=0\}\subset \A_\g$$
is a left-sided ideal in $\A_\g$ called the {\em recursive ideal} of $f,$ c.f. \cite{Ga1}.
This term reflects the fact that each element of $I_f$ represents a recursive relation for $f.$
Function $f$ is $q$-holonomic iff $I_f$ is q-holonomic, which intuitively means "as large as possible"
(and, in particular, non-trivial).
More precisely, a left ideal $I\triangleleft \A_\g$ is q-holonomic
if its homological codimension is at least the rank of $\g,$ that is
\begin{equation}\label{hcd}
hcd(I)=min\{j: Ext^j_{\A_\g}(\A_\g/I,\A_\g)\ne 0\}\geq rank\, \g.
\end{equation}
This definition is a modification of that of \cite{GL1} to functions defined on $\Lambda_\g$ rather than
on $\N^n.$ We denote the recursive ideal of $J_{\g,K}$ by $I_{\g,K}.$

\subsection{The action of the Weyl group}\label{ss_Weyl_act}
$J_{\g,K}(\lambda)$ is equivariant with respect to the Weyl group $W$ action:

\begin{proposition}\label{J-sym}(Proof in Sec. \ref{s_WRT})
$$J_{\g,K}(w\cdot \lambda)=sgn(w)\cdot J_{\g,K}(\lambda),$$
where $sgn(w)=\pm 1$ is the sign of $w\in W.$
\end{proposition}

In particular, $J_{\g,K}(\lambda)$ vanishes for weights $\lambda$ in the boundaries of Weyl chambers.

The Weyl group acts on $F(\Lambda_g,\C[q^{\pm 1/D(\g)}])$ by $w\cdot f(\alpha)=f(w^{-1}\cdot \alpha).$
(The inverse is needed to make sure that this is a left action.)
Additionally, $W$ acts on $\A_g$ via $w\cdot E_\alpha=E_{w\cdot \alpha},$
$w\cdot Q_\alpha=Q_{w\cdot \alpha},$ and the product
$A_\g\times F(\Lambda_g,\C[q^{\pm 1/D(\g)}])\to F(\Lambda_g,\C[q^{\pm 1/D(\g)}])$
is $W$-equivariant. Furthermore, by Proposition \ref{J-sym},
$w\cdot I_{\g,K}=I_{\g,K},$
for every $w\in W.$ We call the $W$-invariant part of the recursive ideal,
$I_{\g,K}^W\triangleleft \A_\g^W,$
the invariant $\g$-recursive ideal of $K.$


\begin{proposition}\label{inv_rec}(Proof in Sec. \ref{proof_inv_rec})
(1) For the unknot $U,$ $$I_{sl(2,\C),U}^W=\langle E+E-(q^{1/2}+q^{-1/2}, EQ+E^{-1}Q^{-1}-q(Q+Q^{-1})\rangle.$$
(2) For the left-sided trefoil, $I_{sl(2,\C),K}^W$ is generated by elements $$q^{5/4}(EQ^{-5}+E^{-1}Q^5)-q^{-7/4}(EQ^{-1}+E^{-1}Q)-q^{-3/4}(Q^5+Q^{-5})+q^{1/4}(Q+Q^{-1}),$$
$q^3(E^2Q^{-6}+E^{-2}Q^6)+(q^{3/2}+q^{-3/2})(E+E^{-1})-(q^{1/3}+q^{-5/2})(EQ^{-6}+E^{-1}Q^6)+(Q^6+Q^{-6})
-2(q+q^{-1}),$\\
$-q^{-7/2}(E^2Q^{-7}+E^{-2}Q^7)+q^{-3}(EQ^{-7}+E^{-1}Q^7)+(q^{-2}-q^{-1})(EQ^{-3}+E^{-1}Q^3)-
q(EQ^{-1}+E^{-1}Q)-(q^{1/2}-q^{-1/2})(Q^3+Q^{-3})+q^{-3/2}(Q+Q^{-1}).$\\
(3) The generators of the invariant recursive ideal of the right-handed trefoil are
obtained from those above after substitution $Q\to Q^{-1},$ $q\to q^{-1}.$
\end{proposition}

\begin{conjecture}
For every $\g$ and $K,$\\
(1) $J_{\g,K}$ is uniquely determined by a finite number of its values together with the recursive relations of $I_{\g,K}.$\\
(2) $J_{\g,K}$ is uniquely determined among $W$-equivariant functions
(i.e. functions satisfying the statement of Proposition \ref{J-sym})
by a finite number of its values together with the recursive relations of $I_{\g,K}^W.$
\end{conjecture}

\subsection{$\A_\g^W$ is a quantization-deformation of the $G$-character variety of the torus.}
For a given complex reductive algebraic group $G$, denote the $G$-character variety of a finitely
generated (discrete) group $\Gamma$ by $X_G(\Gamma),$ c.f. Sec \ref{s_char_var}.
Additionally, denote the connected component of the trivial character in $X_G(\Gamma)$
by $X_G^0(\Gamma).$ (If $G$ is simply connected, for example if $G=GL(n,\C),SL(n,\C), Sp(n,\C),$
then the $G$-character variety is connected, \cite{Ric}, and hence $X_G^0(\Gamma)=X_G(\Gamma)$.)
We often abbreviate $X_G(\pi_1(Y))$ and $X_G^0(\pi_1(Y))$ by $X_G(Y)$ and $X_G^0(Y)$ for a topological
space $Y.$ Goldman proved that for any closed orientable surface $F,$ $X_G(F)$ is a singular
holomorphic symplectic manifold, \cite{Go1}, c.f. Sec. \ref{s_def_quant}.
This symplectic structure defines a Poisson bracket
on the space of holomorphic functions on $X_G^0(F)$ called the Goldman bracket.

\begin{theorem}\label{AgW_quant_s} (Precise statement in Cor. \ref{AgW_quant0} and Thm. \ref{AgW_quant})\\
(1) For every complex reductive algebraic group $G$ and its Lie algebra $\g,$
$\A_g^W$ is a deformation-quantization of $\C[X_G^0(\text{torus})].$\footnote{Note that this statement
implies in particular that up to an isomorphism $X_G^0(\text{torus})$ depends on the Lie algebra of
$G$ only.}\\
(2) For every classical group, $G=GL(n,\C),SL(n,\C),SO(n,\C),Sp(n,\C),$
this deformation-quantization is in the direction of the Goldman bracket.
\end{theorem}

By \cite{FG,Sa}, $\A_{sl(2)}^W$ is the Kauffman bracket skein algebra of the torus.
Therefore, Theorem \ref{AgW_quant_s}(2) generalizes the result of \cite{BFK} for torus to
higher rank classical groups. We will discuss the relations between the present work
and skein modules of higher rank (in particular those of \cite{S-SUn}) in an upcoming \cite{S-skein}.

\subsection{$I_{\g,K}^W$ as a quantization of the $G$-representations of $\pi_1(S^3\setminus K)$.}
By Theorem \ref{AgW_quant_s} we have a $\C$-algebra homomorphism
\begin{equation}\label{ve}
\ve: \A_\g^W\to \C[X_G^0(\Z^2)]
\end{equation}
given by evaluation $q=1$.

Given a knot $K\subset S^3,$ let $M_K$ be the compactification of $S^3\setminus K$ with boundary
torus, $\p M_k=T$. The embedding $\p M_K\hookrightarrow M_K$ defines a homomorphism
$\phi_K: \C[X_G^0(T)]\to \C[X_G^0(M_K)]$ whose kernel we denote by $A_{G,K}.$
We call it the $A_G$-ideal of $K$. The $A_{SL(2,\C)}$-ideal of $K$ determines the
$A$-polynomial of $K$ of \cite{CCGLS}. (That is the motivation for the name "$A_G$-ideal".)

\begin{conjecture}\label{oa}
The zero set of
$\ve(I_{\g,K}^W)\triangleleft \C[X_G^0(T)]$ is the closure of the image of $X_G^0(M_K)\to
X_G^0(T).$ Equivalently,
$$\sqrt{\ve(I_{\g,K}^W)}=A_{G,K},$$ where $\sqrt{\cdot}$ denotes the nil-radical.
\end{conjecture}

Let $\Z^2=\langle L,M \rangle.$ By Theorem \ref{AgW_quant_s},
$\C[X_{SL(2,\C)}(\Z^2)]=\C[E^{\pm 1}, Q^{\pm 1}]^{\Z/2},$
where $\Z/2=\{e,\imath\},$
$$\imath(E)=E^{-1},\quad \imath(Q)=Q^{-1}.$$
Under the isomorphism which will be defined in (\ref{isoms}),
the regular function $\tau_{a,b}\in \C[X_{SL(2)}(\Z^2)],$
$\tau_{a,b}([\rho])= tr \rho (L^aM^b),$  corresponds
to $E^a Q^b+E^{-a}Q^{-b},$ for any $a,b\in \Z.$

\begin{proposition}\label{IvA}(Proof in Sec \ref{proof_Iva})\\
(1) For the unknot $U,$
$$\ve(I_{sl(2,\C),U}^W)=\langle E+E^{-1}-2, EQ+E^{-1}Q^{-1}-Q-Q^{-1}\rangle=A_{sl(2,\C),U}\triangleleft
\C[X_{SL(2,\C)}(\Z^2)].$$
(2) For the left-handed trefoil,
$$\ve(I_{sl(2,\C),K}^W)=\langle w(Q^2-Q^{-2}), w^2,
w(EQ^4-E^{-1}Q^{-4}) \rangle,$$
$$A_{sl(2,\C),K}=\langle w(Q-Q^{-1}),w(E-E^{-1}),w(EQ^{-1}-E^{-1}Q)\rangle,$$
where $w=E^{-1}Q^3(E-1)(EQ^{-6}+1).$
(3) For the right-handed trefoil, one needs to change $Q$ to $Q^{-1}$ in the formulas above.
\end{proposition}

\begin{corollary}\label{IvAcor}
For the unknot and for the trefoil Conjecture \ref{oa} holds.
However, $\ve(I_{sl(2,\C),K}^W) \varsubsetneq A_{sl(2,\C),K}$ for the trefoil.
\end{corollary}

\begin{proof} For the unknot the statement is obvious.
For the trefoil, observe that the square of every generator of $A_{sl(2,\C),K}$ is divisible by $w^2$
and hence contained in $\ve(I_{sl(2,\C),U}^W.$ Therefore, $\sqrt{\ve(I_{\g,K}^W)}=A_{G,K}.$
On the other hand, $Q-1$ does not belong to the ideal $\langle Q^2-Q^{-2},w, EQ^4-E^{-1}Q^{-4}\rangle$,
since $(E,Q)=(1,-1)$ belongs to the zero set of that ideal.
Consequently, $w(Q-Q^{-1})\not\in \ve(I_{\g,K}^W).$
\end{proof}

Corollary \ref{IvAcor} disproves \cite[Conj. 2]{LeTB}.

\begin{theorem}\label{def=char} (Proof in Sec. \ref{def=char_p})
Conjecture \ref{oa} for a given $K$ and $\g$ implies that the characteristic and deformation varieties of
Garoufalidis coincide and, in particular, AJ conjecture of \cite{Ga2} holds and \cite[Question 1]{Ga2}
has affirmative answer.
\end{theorem}

\subsection{Acknowledgements} We would like to thank Charlie Frohman and Thang T. Q. Le for helpful discussion.


\section{WRT knot functions and quantum Weyl algebras}
\label{s_WRT}


For every $\g$-module $V,$ $V[[h]]=V\otimes \C[[h]]$ is a module over the quantum
group $U_h(\g),$ \cite{CP,KS}. Let $K'$ be a $1$-tangle obtained by cutting a knot $K$ open.
Reshetikhin-Turaev construction associates with $K'$ colored by a representation
$V$ of $\g$ (or $V[[h]]$ of $U_h(\g)$) a morphism of $U_q(\g)$-modules $V[[h]]\to V[[h]]$ which is in
the center of
$End_{U_h(\g)}(V[[h]]).$ (This is related to the fact that all $1$-tangles commute under composition.)
Consequently, if $V$ is irreducible then, by Schur's Lemma, the WRT invariant of $K'$ is a scalar
multiple of the identity. We denote that scalar by $WRT_{\g,V}(K').$ It lies in
$\C[q^{\pm 1}]\subset \C[[h]],$ where $q=e^h,$
and $$WRT_{\g,V}(K)=WRT_{\g,V}(U)\cdot WRT_{\g,V}(K'),$$
where the WRT-invariant of the unknot is the quantum dimension
of $V$,
$$WRT_{\g,V}(U)=dim_q(V)= \frac{\sum_{w\in W} sgn(w)q^{(\lambda+\rho,w(\rho))}}
{\sum_{w\in W} sgn(w)q^{(\rho,w(\rho))}}$$
and $\rho$ is the half-sum of positive roots.
Rozansky (for $sl(2)$ in \cite{Ro}), Le for $sl(2),$ \cite{HL}, and for all $\g$,
\cite[Lemma 7.7]{GL1}, proved that Reshetikhin-Turaev definition of $WRT_{\g,V}(K')$ makes sense
verbatim for all Verma modules $V$ despite the fact that they are infinite dimensional.
(This is not obvious, since Reshetikhin-Turaev construction involves quantum traces, which
in case of infinite dimensional modules involve sums in $\C[q^{\pm 1}]$ which are a priori infinite.
One has to prove that all but finitely many summands in all these sums vanish.)

For every Verma module $\lambda\in \Lambda_\g$, we define
\begin{equation}\label{Verma_def}
WRT_{\g,M(\lambda)}(K)=dim_q(M(\lambda))\cdot WRT_{\g,M(\lambda)}(K')
\end{equation}
and, by (\ref{WRT_J}), $J_{\g,K}(\lambda)=WRT_{\g,M(\lambda-\rho)}(K).$

By \cite[Theorem 8]{GL1}, $WRT_{\g,M(\lambda)}(K')$ is $q$-holomorphic.
Since the product of $q$-holomorphic functions is $q$-holomorphic,
$J_{\g,K}(\lambda)$ is $q$-holomorphic as well.

\subsection{Example: $sl(n)$-quantum invariants of the unknot}
\label{ss_sln_unknot}
Let $sl(n)$ be the algebra of traceless $n\times n$ matrices.
Consider the standard Cartan subalgebra of $sl(n)$ composed of diagonal matrices:
$$\h=\left\{\sum a_iE_{ii}: \sum a_i=0\right\}.$$
The weights $\alpha_i:\h\to \C,$ such that $\alpha_i(E_{jj})=\delta_{ij},$ generate the weight
lattice of $sl(n)$ and the Killing form on $\h^*$ is given by

\begin{equation}\label{killing_sln}
(\alpha_i,\alpha_j)=\begin{cases} \frac{n-1}{n} & i=j\\
-\frac{1}{n} & i\ne j,\\
\end{cases}
\end{equation}
c.f. \cite[Formula 15.2]{FH}.
The Weyl group $W=S_n$ permutes the weights $\alpha_1,...,\alpha_n.$
The positive roots are $\alpha_i-\alpha_j,$ for $i>j,$ and $$\rho=\frac{n-1}{2}\alpha_1+
\frac{n-3}{2}\alpha_2+... -\frac{n-3}{2}\alpha_{n-1}-\frac{n-1}{2}\alpha_n.$$
Let $E_i=E_{\alpha_i}$, $i=1,...,n.$
By (\ref{unknot}),
\begin{equation}\label{E_unknot}
\sum_{i=1}^n E_iJ_{sl(n,\C),U}=\frac{1}{S} \sum_{w\in W} \left(sgn(w)q^{(\lambda,w(\rho))} \sum_{i=1}^n
q^{(\alpha_i,w(\rho))}\right),
\end{equation}
where $$S={\sum_{w\in W} sgn(w)q^{(\rho,w(\rho))}}.$$
Since $(\alpha_i,\rho)=\frac{1}{2}(n+1-2i),$
the second sum in (\ref{E_unknot}) is equal to
$$\sum_{i=1}^n q^{(\alpha_i,w(\rho))}=\sum_{i=1}^n q^{(w^{-1}(\alpha_i),\rho)}=
\sum_{i=1}^n q^{(\alpha_i,\rho)}=\sum_{i=1}^n q^\frac{n+1-2i}{2}=[n],$$ for every $w\in S_n,$
where $[n]$ is the $n$-th quantum integer, $[n]=\frac{q^{n/2}-q^{-n/2}}{q^{1/2}-q^{-1/2}}.$

\begin{corollary}\label{rec_unknot}
$\sum_{i=1}^n E_i - [n]$ belongs to the $sl(n)$-recursive ideal of
the unknot.
\end{corollary}

\noindent{\bf Proof of Proposition \ref{J-sym}} (Suggested by T. Le):
By (\ref{WRT_J}) and (\ref{Verma_def}),
$$J_{\g,K}(\alpha)=J_{\g,U}(\alpha)\cdot WRT_{\g,M(\alpha-\rho)}(K').$$
By (\ref{unknot}), $$J_{\g,U}(w\cdot \alpha)=sgn(w)\cdot J_{\g,U}(\alpha).$$ Therefore,
it is enough to prove that $WRT_{\g,M(\alpha-\rho)}(K')$ is invariant under the action of $W$
on $\alpha.$

For every positive element $w\in W$ in Bruhat ordering (i.e. a product of reflections with respect of
positive roots)
and for every $\alpha\in \Lambda_g,$ $M(\lambda-\rho)$ is a submodule of $M(w\cdot \lambda-\rho),$ c.f.
\cite[Ch V.9 Problem 12]{Kn}.\footnote{In Knapp's book the Verma module $M(\lambda)$ is denoted by
$V(\lambda+\rho)$.}
This implies that $$WRT_{\g,M(\lambda-\rho)}(K')=WRT_{\g,M(w\cdot\lambda-\rho)}(K').$$
Since positive elements in $W$ generate $W$, the proof is completed.
\Box

\begin{lemma}\label{lin_indep}
Operators $Q_{\alpha} E_{\beta},$ for $\alpha,\beta\in \Lambda_\g,$ are linearly independent.
\end{lemma}

\begin{proof}
Suppose that
\begin{equation}\label{lin_dep}
\sum_{\alpha,\beta} c_{\alpha,\beta} Q_{\alpha} E_{\beta}=0
\end{equation}
and $c_{\alpha_0,\beta_0}\ne 0$ for some $\alpha_0,\beta_0\in \Lambda_\g$ such that
\begin{equation}\label{max_beta}
(\beta_0,\beta_0)=max\{(\beta,\beta): c_{\alpha,\beta}\ne 0\ \text{for some $\alpha$}\}
\end{equation}
and
\begin{equation}\label{max_alpha}
(\alpha_0,\alpha_0)=max\{(\alpha,\alpha): c_{\alpha,\beta_0}\ne 0\}.
\end{equation}
Fix an integer $N>max\left\{\frac{(\beta,\beta_0)}{(\alpha_0,\alpha_0)}: c_{\alpha_0,\beta}\ne 0\right\}$
and let $f: \Lambda_\g \to \C[q^{\pm 1/D(\g)}],$
$$f(v)=\begin{cases} 1 & \text{if $v=kN\alpha_0$ for some $k\in \Z$}\\
0 & \text{otherwise}.\\ \end{cases}$$
Then the value of $\sum_{\alpha,\beta} c_{\alpha,\beta} Q_{\alpha} E_{\beta}f$ at
$-\beta_0+kN\alpha_0$ is zero.
On the other hand,
$$\sum_{\alpha,\beta} c_{\alpha,\beta} Q_{\alpha} E_{\beta}f(-\beta_0+kN\alpha_0)=
\sum_{\alpha,\beta} c_{\alpha,\beta} q^{(\alpha,-\beta_0+kN\alpha_0)}f(\beta-\beta_0+kN\alpha_0).$$
By the definitions of $f$ and $N$, the sum on the right equals
$$\sum_{\alpha} c_{\alpha,\beta_0} q^{(\alpha,-\beta_0+kN\alpha_0)},$$ and
by (\ref{max_alpha}), its leading term is
$$c_{\alpha_0,\beta_0} q^{(\alpha_0,-\beta_0+\alpha_0 kN)}.$$
It grows exponentially with $k$ -- a contradiction.
\end{proof}

\noindent{\bf Proof of Proposition \ref{presentation}:}
Since relations (\ref{qWeyl_rel}) are obviously satisfied by the operators $E_\alpha$ and $Q_\beta,$
$\alpha,\beta\in \Lambda_\g,$ it is enough to prove that all other relations between these operators
follow from (\ref{qWeyl_rel}).
Let $P$ be a polynomial in $E_\alpha$'s and $Q_\beta$'s, $\alpha,\beta\in \Lambda_\g,$
which equals to the zero operator on $F(\Lambda_\g,\C[q^{\pm 1/D(\g)}]).$
Relations (\ref{qWeyl_rel}) make possible to express $P$ as a sum
$$P=\sum_{\alpha,\beta} c_{\alpha,\beta} Q_{\alpha} E_{\beta},$$
over $\Lambda_\g\times \Lambda_\g,$ with $c_{\alpha,\beta}\in \C[q^{\pm 1/D(\g)}].$
By Lemma \ref{lin_indep}, all $c_{\alpha,\beta}$'s in the above sum vanish.
Hence the relation $P=0$ is a consequence of relations (\ref{qWeyl_rel}).
\Box

%
\subsection{Proof of Proposition \ref{inv_rec}:}\label{proof_inv_rec}
%
(1) By \cite{FG,Sa}, $\A_{sl(2,\C)}^W$ is isomorphic to the Kauffman bracket skein module of the torus
and, by \cite{Ga1}, $I_{sl(2,\C),K}^W$ corresponds to the orthogonal ideal under that isomorphism.
More specifically, the $p/q$-torus knot on the torus corresponds to $(-1)^{p+q}q^{-ab/4}(E^pQ^q+E^{-p}Q^{-q}),$  \cite[Fact 4]{Ga1},
and the $t$ of \cite{FG, Ge} is our $q^{1/4}.$
The orthogonal ideal of the unknot was computed in \cite{FG}
It is generated by two elements:
$$\text{longitude} + (t^2+t^{-2})\quad \text{and}\quad (1,1)\text{-curve} +t^3\text{meridian}.$$
Therefore, $I_{sl(2,\C),U}^W$ is generated by $E+E^{-1}-(q^{1/2}+q^{-1/2})$ and $EQ+E^{-1}Q^{-1}-q(Q+Q^{-1}).$\\
(2) The orthogonal ideal of the left and right handed trefoil was computed in \cite{Ge}.
The result can be summarized as follows:
For $p,q$ coprime, let $(p,q)$ be the $p/q$-curve on the torus $T$ considered
as an element of the Kauffman bracket skein module of $T\times I,$
so that $(1,0)$ is the longitude and $(0,1)$ meridian with respect to the embedding $T=\p M_K\subset M_K.$
For $p,q$ such that
$gcd(p,q)=n\geq 0,$ let $(p,q)=T_n((p/n,q/n)),$ where $T_n(x)$ is the $n$-th Chebyshev polynomial:
$T_0(x)=2,$
$T_1(x)=x,$ $T_{n+1}(x)=x T_n(x)-T_{n-1}.$
Then the Kauffman bracket peripheral ideal of the left-handed trefoil $K$ is generated by elements:\footnote{We have independently verified that these polynomials generate $P_{sl(2,\C),K}.$ Please note the plus sign in the second term of the third generator, which is missing in \cite{Ge}.}\\
$(1,-5)-t^{-8}(1,-1)+t^{-3}(0,5)-t(0,1),$\\
$(2,-6)-(t^6+t^{-6})(1,0)+(t^4+t^{-4})(1,-6)+(0,6)-2(t^4+t^{-4}),$\\
$(2,-7)+t^{-5}(1,-7)+(t^{-5}-t^{-1})(1,-3)-t^5(1,-1)+(t^2-t^{-2})(0,3)-t^{-6}(0,1).$
Now the statement follows as in (1).

%
\subsection{Proof of Proposition \ref{IvA}:}\label{proof_Iva}
%
(1) By Proposition \ref{inv_rec},
$$\ve(I_{\g,U}^W)=\langle E+E^{-1}-2, EQ+E^{-1}Q^{-1}-Q-Q^{-1}\rangle.$$

It remains to prove that this ideal coincides with $A_{sl(2,\C),U}\triangleleft
\C[X_{SL(2,\C)}(\Z^2)].$
It is easy to check that $E+E^{-1}-2$ and $EQ+E^{-1}Q^{-1}-Q-Q^{-1}$ belong to $A_{sl(2,\C),U}^W.$
We claim that these two elements generate $A_{sl(2,\C),U}.$
Let $x_k=EQ^k+Q^{-1}Q^{-k}-(Q^k+Q^{-k}).$ Since $x_0,x_1\in A_{sl(2,\C),U}$ and $x_{k+1}=(Q+Q^{-1})x_k-x_{k-1},$ $x_k\in A_{sl(2,\C),U},$ for all $k.$

Any element of that ideal can be reduced by $E+E^{-1}-2$ to a polynomial in
$\C[E^{\pm 1},Q^{\pm 1}]^{\Z/2},$ of span at most
$2$ in $E.$ It is easy to see that any such element is of the form
$Ep+q+E^{-1}\imath(p),$ where $p,q\in \C[Q^{\pm 1}].$
Therefore any $z\in A_{sl(2,\C),U}$ can be presented as
$$\sum_k c_k x_k +w(Q),$$
where $c_k\in \C$ and $w(Q)\in \C[Q^{\pm 1}].$
Since all $x_k$'s are in $A_{sl(2,\C),U},$ $w(Q)\in A_{sl(2,\C),U}.$
However, since the meridian of the unknot can be mapped to $\left(\begin{matrix} m & 0\\
0 & m^{-1}\\ \end{matrix}\right)$ for every $m\in \C^*,$ $w(Q)=0.$
Hence, every element of $A_{sl(2,\C),U}$ is a linear combination of $x_k$'s modulo
$E+E^{-1}-2.$

(2) The generators of the invariant recursive ideal of the left-handed
trefoil listed in Proposition \ref{inv_rec}(2) are equal to
$w(Q^2-Q^{-2}), w(1-E^{-1})(EQ^{-3}+Q^3), w(EQ^{4}-E^{-1}Q^{-4})$ for $q=1.$
In order to compute $A_{sl(2,\C),K}$ observe that since $(E-1)(E Q^{-6}+1)$ is the $A$-polynomial of the left handed-trefoil,
$$A_{sl(2,\C),K}^W=(w \cdot \C[E^{\pm 1},Q^{\pm 1}])\cap
\C[E^{\pm 1},Q^{\pm 1}]^{\Z/2}.$$ Since $\imath(w)=-w,$ every element of $A_{sl(2,\C),K}^W$ is of the form
$w\cdot p,$ where $\imath(p)=-p.$ Hence $p$ is a sum of monomials of the form $r_{a,b}=E^aQ^b-E^{-a}Q^{-b}.$
Since $$r_{a+1,b}=(E+E^{-1})r_{a,b}-r_{a-1,b},\quad r_{a,b+1}=(Q+Q^{-1})r_{a,b}-r_{a,b-1},$$
$p$ is a linear combination of $E-E^{-1},Q-Q^{-1},EQ-E^{-1}Q^{-1}$ with coefficients in
$\C[E^{\pm 1},Q^{\pm 1}]^{\Z/2}.$

\subsection{Proof of Theorem \ref {def=char}}\label{def=char_p}
The characteristic variety of Garoufalidis is the Zariski closure of the zero set
$$Z(A_{G,K})\subset Z(\C[\Lambda_\g])=(\C^*)^n\subset \C^n,$$
where $n$ is the rank of $\g.$
Similarly, the deformation variety is the closure of
$$Z(\ve(I_{\g,K}))\subset Z(\C[\Lambda_\g])=(\C^*)^n\subset \C^n.$$
Therefore, it is enough to prove that Conjecture \ref{oa} implies that for every $K$ and $\g$
$$\sqrt{\ve(I_{\g,K})}=\sqrt{A_{G,K}}.$$

To show the inclusion "$\subset$", it is enough to prove that $\ve(I_{\g,K})\subset \sqrt{A_{G,K}}$.
For each $g\in I_{\g,K},$ the element $\ve(g)$ is a root of the polynomial
$\prod_{w\in W} (x-w\cdot \ve(g))=\sum c_k x^k$ with coefficients
$c_k\in \ve(I_{\g,K}^W)$ for $k=0,...,N-1,$ where $N=|W|$ and $c_N=1.$
Assuming that Conjecture \ref{oa} holds, $\ve(I_{\g,K}^W)\subset A_{G,K},$
Consequently, $\ve(g)^N=-\sum_{k=0}^{N-1} c_k \ve(g)^k\in A_{G,K}$ and $\ve(g)\in\sqrt{A_{G,K}}.$

To show the inclusion "$\supset$", it is enough to prove that $A_{G,K}\subset \sqrt{\ve(I_{\g,K})}.$
Each $h\in A_{G,K}$ is a root of the polynomial
$\prod_{w\in W} (x-w\cdot h)=\sum c_k x^k$ with coefficients
$c_k\in A_{G,K}^W= \sqrt{\ve(I_{\g,K}^W)}.$ Consequently, $h^N=-\sum_{k=0}^{N-1} c_k h^k\in
\sqrt{\ve(I_{\g,K}^W)}.$ Hence, $h\in \sqrt{\ve(I_{\g,K}^W)}.$
\Box

Despite the fact that $A_{G,K}^W$ is equal its nil-radical, $A_{G,K}$ is often not equal to its
nil-radical. Indeed, for the unknot, $(l-l^{-1})^2\in A_{sl(2),U}$ but $(l-l^{-1})\not \in A_{sl(2),U}$!

Note that the conclusion of Theorem \ref{def=char} is stronger than the
AJ-conjecture of Garoufalidis for $\g=sl(2),$ \cite[Conjecture 1]{Ga2}, and
its version for higher rank Lie algebras, \cite[Question 1]{Ga2}.


\section{Character varieties}
\label{s_char_var}


\subsection{Introduction} Let $G$ be a complex reductive algebraic group. If $\Gamma$ is a (discrete) group generated by
$\gamma_1,...,\gamma_n$ then the set of homomorphisms $Hom(\Gamma,G)$
can be identified with the set of points $(\rho(\gamma_1),...,\rho(\gamma_n))\in G^n$
taken over all representations $\rho: \Gamma\to G.$ It is an algebraic set which up to an isomorphism
does not depend on the choice of generators of $\Gamma.$
The group $G$ acts on $Hom(\Gamma,G)$ by conjugating representations and
the categorical quotient of that action, $$X_G(\Gamma)=Hom(\Gamma,G)//G$$
is called the $G$-character variety of $\Gamma.$ In simple words $X_G(\Gamma)$ is an algebraic set
together with a map $\pi: Hom(\Gamma,G)\to X_G(\Gamma)$ which is constant on all $G$-orbits and has the universal property that every map $Hom(\Gamma,G)\to Y$ which is constant on all $G$-orbits factors through $\pi.$

If $\Gamma$ is the fundamental group of a topological space $X,$ then $X_G(\Gamma)$ is called the
$G$-character variety of $X$ and it is abbreviated by $X_G(X).$

\begin{proposition}(\cite{S-char})\label{char_gen}
For $G=SL(n,\C),$ $O(n,\C),$ $Sp(2n,\C),$ let $\tau_\gamma : X_G(\Gamma)\to \C$ be defined as
$\tau_\gamma ([\rho])= tr(\rho(\gamma))$ for $\rho: \Gamma\to G\to GL(V),$ where $V$ is the defining
representation of $G$ (The faithful representation of the smallest dimension.)
Then the algebra $\C[X_G(\Gamma)]$ is generated by $\tau_{\gamma}$ for all $\gamma\in \Gamma.$
\end{proposition}

Proposition \ref{char_gen} does not hold for $SO(n,\C),$ \cite{S-char}.

Goldman proved that for every complex reductive algebraic group $G$ and any closed orientable surface $F,$
$X_G(F)=X_G(\pi_1(F))$ is a singular holomorphic symplectic manifold. More specifically, let
$X_G^{irr}(F)$ be the set of conjugacy classes of all representations
$\rho: \pi_1(F)\to G$ such that $\rho(\pi_1(F))$ is not contained in a proper connected algebraic subgroup of $G$. $X_G^{irr}(F)$ is an open subset of $X_G(F)$ and a smooth (complex)
manifold. Goldman defines holomorphic symplectic form on $X_G^{irr}(F).$
His construction utilizes the fact that the tangent space $T_{[\rho]} X_G(F)$ at $[\rho]\in X_G^s(F)$ represented by a representation $\rho: \pi_1(F)\to G$ is canonically isomorphic to $H^1(F, Ad_\rho \g).$
Let $B$ be a non-degenerate symmetric bilinear form, $B:\g \times \g \to \C,$
invariant under the adjoint $G$ action. For classical groups, the standard choice is the
trace form, $B(X,Y)= tr(XY),$ where the trace is defined by the embedding $G\subset GL(V)=GL(n,\C),$
for the defining representation $V$ of $G$.
The induced cup product
\begin{equation}\label{omega}
\omega: H^1(F, Ad_\rho \g)\times H^1(F, Ad_\rho \g) \stackrel{\cup}{\longrightarrow} H^2(F,\C)
\stackrel{\cap [F]}{\longrightarrow} \C,
\end{equation}
defines a symplectic form on $T_{[\rho]}X_G^{irr}(F).$
Goldman proves by an argument from gauge theory that
$\omega$ is closed, \cite{Go1}. This is Goldman's symplectic form.


\subsection{Character varieties of the torus}
Character varieties are usually very difficult to describe as solution
sets of explicit systems of polynomial equations. Even an explicit description of
$X_G(\Z^2)$ is difficult in general, since the number of connected and irreducible components of
this set is unknown. If $G$ is simply connected, for example $G=SL(n,\C)$ or $Sp(n,\C),$
then $X_G(\Z^2)$ is connected, by \cite{Ric}. However, $X_G(\Z^2)$ may be not connected in general:
Points of $X_{G}(\Z^2)$ classify flat principal $G$-bundles over the torus.
If $E_\rho$ is the bundle corresponding to $[\rho]\in X_G(\Z^2)$ then the second
obstruction class to the existence of a global section of $E_\rho$ lies in
$H^2(\Z^2,\pi_1(G))=\pi_1(G),$ with the action of $\Z^2$ on $\pi_1(G)$ given by $\rho,$ c.f. \cite{Go2}.
The obstruction map $F: X_G(\Z^2)\to \pi_1(G)$ is constant on connected components. Goldman
conjectures that $F$ maps bijectively connected components of $X_G(\Z^2)$ onto $\pi_1(G),$ for
all semi-simple algebraic groups, \cite{Go2}.

As before, let $X_G^0(\Gamma)$ be the connected component of the trivial character.
The proof of the following statement appears in \cite{Th}. For the convenience of the reader we
include the proof below.

\begin{theorem}\label{char_var}(\cite{Th})
For any complex reductive algebraic group $G$ and its Cartan subgroup
(a maximal complex torus) $\T$, the map
$$\T^2=Hom(\Z^2,\T)\to Hom(\Z^2,G)\to Hom(\Z^2,G)//G=X_G(\Z^2)$$ factors through
an isomorphism
$\chi: \T^2/W\to X_G^0(\Z^2)$, where the Weyl group $W$ acts diagonally on $\T\times \T.$
\end{theorem}

\begin{proof}
Let us first give an elementary proof that $\chi$ is onto for $G=GL(n,\C)$ and $SL(n,\C)$.
For any commuting matrices
$g_1,g_2\in G$ there is $h\in G$ such that $hg_1h^{-1},hg_2h^{-1}$ are upper triangular.
Furthermore, $h$ can be chosen so that the entries of $hg_1h^{-1},hg_2h^{-1}$ above diagonal
are uniformly arbitrarily small. Therefore, every representation $\rho:\Z^2\to G$
has an infinite sequence of conjugates approaching some $\phi: \Z^2\to \T\subset G$ in the classical
topology and, hence, in the Zariski topology as well. Since all points are closed in $X_G(\Z^2),$ every
point in it is represented by some $\phi: \Z^2\to \T.$

In the proof of $\chi$ being onto for every reductive $G$ we follow an argument of \cite{Th}.
Denote the equivalence class of $\rho:\Z^2\to G$ in $X_G(\Z^2)$ by $[g_1,g_2]$, where
$g_1=\rho(1,0),$ $g_2=\rho(0,1).$
Let $h$ be a regular semisimple element of $G$ and let $\T$ be its centralizer.
Then every element in some open neighborhood of $h$ in $G$ is conjugated to a regular
element in $\T.$ Hence $[h,e]$ has an open neighborhood $U\subset X_G(\Z^2)$ in complex topology
whose every element is represented by $[gh'g^{-1},k]=[h',g^{-1}kg],$ where $h'\in \T$ and $h'$ is regular.
Since $\T$ is the centralizer of $h',$ $g^{-1}kg\in \T.$ Hence we proved that the image of $\chi$ contains
the open set $U$ in complex topology. Consequently, the image of $\chi$ is dense in Zariski topology.
However, by \cite{Ric}, the connected component $Hom^0_G(\Z^2)$ of the trivial homomorphism in $Hom(\Z^2,G)$
is irreducible for every connected reductive group $G.$ Therefore, $\chi$ is onto.

To prove that $\chi$ is $1$-$1,$ we need to show that if $g_1,...,g_N, g_1',...,g_N'\in \T$ and
$(g_1',...,g_N')=g(g_1,...,g_N)g^{-1}$ for some $g\in G$ then
$(g_1',...,g_N')=w(g_1,...,g_N),$ for some $w\in W$ acting on $\T$.
We follow an argument of Borel, \cite{Bo}, and Thaddeus, \cite{Th}:
The centralizer of $g_1,...,g_N,$ $Z(g_1,...,g_n)\subset G$ is a reductive group
by \cite[26.2A]{Hu} since the proof there is valid not only for a subtorus but for any
subset. $\T$ and $g^{-1}\T g$ are maximal tori in $Z(g_1,...,g_n)$
and, therefore, $\T$ is conjugate to $g^{-1}\T g$ by some $h\in Z(g_1,...,g_n).$
Then $hg\in N(\T)$ represents $w\in W$ which sends $(g_1,...,g_N)$ to $(g_1',...,g_N').$
\end{proof}

More generally, one can prove that $\chi: \T^n/W\to X_G^0(\Z^n)$ is an embedding for every $n.$

Let $\Lambda_\g$ be the weight lattice of the Lie algebra $\g$ of $G.$
Since every weight $\alpha\in \Lambda_\g$ is a homomorphism $\alpha: \T\to \C^*$ and a regular
function on $\T$, there is an natural map $L: \Lambda_\g\to \C[\T].$ Extending it
additively to the group ring of $\Lambda_\g$ we get a $\C$-algebra homomorphism
$L: \C[\Lambda_\g]\to \C[\T].$

\begin{lemma}
$L: \C[\Lambda_\g]\to \C[\T]$ is an isomorphism of $\C$-algebras.
\end{lemma}

\begin{proof}
Let $\alpha_1,...,\alpha_n: \T\to \C^*$ be weights of a faithful representation $V$ of $G.$
Since the intersection of the kernels of these group homomorphisms is trivial,
$(\alpha_1,...,\alpha_n)$ embeds $\T$ into $(\C^*)^n.$ Consequently, $\alpha_i$'s generate
$\C[\T]$ and $L$ is onto. To show that $L$ is $1-1$, note that $L$ embeds the group $\Lambda_\g$
into the multiplicative group $(\C[\T])^*$ of the ring $\C[\T].$
Since $\C[\T]\simeq \C[x_1^{\pm 1},...,x_n^{\pm 1}],$ the elements of $(\C[\T])^*=$
\mbox{$\langle x_1^{\pm 1},...,
x_n^{\pm 1}\rangle$} are linearly independent in $\C[\T].$ Hence $L$ is $1-1.$
\end{proof}

Consequently, $L\otimes L$ is an isomorphism between $\C[\Lambda_\g^2]=
\C[\Lambda_\g]\otimes \C[\Lambda_\g]$ and $\C[\T^2]$
restricting to an isomorphism
$\C[\Lambda_\g^2]^W\to \C[\T^2]^W=\C[\T^2/W].$

\begin{corollary}\label{x_0}
$\C[X_G^0(\Z^2)]\simeq \C[\Lambda_\g^2]^W.$ Consequently, the algebraic variety $X_G^0(\Z^2)$
depends on the Lie algebra of $G$ only.
\end{corollary}

\begin{example}\label{example_sln}
If $G=SL(n,\C)$ then $\T=\{(x_1,...,x_n)\in (\C^*)^n: x_1\cdot ...\cdot x_n=1\}$ and
$X_{SL(n,\C)}(T)=\T^2/S_n$ where
$\sigma(x_1,...,x_n,y_1,...,y_n)=(x_{\sigma(1)},...,x_{\sigma(n)},
y_{\sigma(1)},..., y_{\sigma(n)}),$ for $\sigma\in S_n.$
\end{example}

\begin{corollary} For $G=SL(n,\C),$ $SO(n,\C)$, $O(n,\C),$ $Sp(2n,\C),$
the algebra $\C[X_G^0(\Gamma)]$ is generated by $\tau_{\gamma}$ for all $\gamma\in \Gamma.$
\end{corollary}

\begin{proof} Since the embedding $X_G^0(\Gamma)\hookrightarrow X_G(\Gamma)$ induces an epimorphism
$\C[X_G(\Gamma)]\to \C[X_G^0(\Gamma)],$ the statement follows immediately from Proposition \ref{char_gen} for $G=SL(n,\C),$ $O(n,\C),$ $Sp(2n,\C).$ For $G=SO(n,\C)$ the statement follows Corollary \ref{x_0}.
\end{proof}

%

\section{Deformation-quantizations}\label{s_def_quant}

%

If $(M,\omega)$ is a holomorphic symplectic manifold then the space $\cH(M)$ of holomorphic
functions on it is a Poisson algebra, i.e. a commutative algebra together with a Poisson bracket: if $f,g\in \cH(M)$ then
\begin{equation}\label{hamiltonian_bracket}
\{f,g\}=-V_f(g),
\end{equation}
where $V_f$ is the hamiltonian vector field of $f$ defined by condition
$\omega(V_f,W)=W(f)$ for every vector field $W$.

Theorem \ref{G-brack} in Sec. \ref{s_Goldman} implies:

\begin{remark}\label{remark}
For every classical group $G$ and for every closed orientable surface $F$, the Poisson bracket on $\cH(X_G(F))$ restricts to
a Poisson bracket on $\C[X_G(F)].$
\end{remark}

For any $\C$-subalgebra $R\subset \C[[h]],$ let $\C_0$ be $\C$ considered as an $R$-module via
the homomorphism $\ve: R\to \C,$ $\ve(h)=0.$ Let $B$ be an associative $R$-algebra, such that
the $R$-submodule of $B$ generated by $Ker \ve$ is an ideal. In this case, $B/(Ker \ve)=B\otimes_R \C_0$. Let us assume that this ring is commutative.

Since $(B\otimes_R \C[[h]])/(h)=B\otimes_R \C_0,$
for any $x,y\in B,$ $x\cdot y-y\cdot x$ is divisible by $h$ in $B\otimes_R \C[[h]]$
and, consequently,
\begin{equation}\label{Poisson}
\{x,y\}= \frac{1}{h}\left(x'\cdot y'-y'\cdot x'\right) +h B\otimes_R \C[[h]]
\end{equation}
defines a unique element in $B\otimes_R \C_0.$
It is easy to check that $\{x,y\}$ depends on the coset values of $x$ and $y$ in $B\otimes_R \C_0$ only
and, therefore, $\{\cdot,\cdot\}$ descends to a bracket on $B\otimes_R \C_0.$ Furthermore, it
is a Poisson bracket.

Let $A$ be a commutative algebra with a Poisson bracket $\{\cdot, \cdot\}: A\times A\to A.$ We say that
$B$ as above is a deformation quantization of $A$ in the direction of $\{\cdot, \cdot\}$ if there is
an isomorphism of Poisson algebras $\Psi:B\otimes_R \C_0\to A.$

(Often, deformation-quantization is defined more restrictively with the conditions:
$R=\C[[h]]$ and $B$ is topologically $R$-free.)

%
%

Given the embedding $R=\C[q^{\pm 1/D(\g)}]\subset \C[[h]],$ $q=e^h,$
$\C_0$ is the $\C[q^{\pm 1/D(\g)}]$-module $\C$ via the homomorphism $q\to 1.$
There is an isomorphism $$\eta: \A_\g\otimes \C_0\to \C[\Lambda_\g^2],$$
$\eta(E_\alpha)=(\alpha,0),$ $\eta(Q_\alpha)=(0,\alpha),$ for $\alpha\in \Lambda_\g,$
which restricts to an isomorphism $$\eta: \A_\g^W\otimes \C_0\to \C[\Lambda_\g^2]^W.$$

Therefore,
\begin{equation}\label{isoms}
\Theta: \A_\g^W\otimes \C_0\stackrel{\eta}{\longrightarrow} \C[\Lambda_\g^2]^W
\longrightarrow \C[\T^2/W]\simeq \C[X_G^0(\Z^2)]
\end{equation}
is an isomorphism as well.


\begin{corollary}\label{AgW_quant0}
$\A_\g^W$ together with (\ref{isoms}) is a deformation-quantization
of $\C[X_G^0(\Z^2)].$
\end{corollary}

Now Theorem \ref{AgW_quant_s} can be stated more precisely as

\begin{theorem}\label{AgW_quant} (Proof in Sec. \ref{s_poisson})
For classical Lie algebras, $\g=sl(n), sp(2n,\C),$ and $so(n,\C),$
the above deformation-quantization is in the direction of the Goldman bracket.\footnote{
By Remark \ref{remark}, the Poisson bracket on $\C[X_G(T)]$ is well defined.}
\end{theorem}

We conjecture that the above statement holds for the exceptional Lie algebras $\g$ as well.


\section{Goldman bracket}
\label{s_Goldman}


\begin{proposition}\label{G-brack}
(1) For $G=SL(n,\C)$ the Goldman bracket is given by
\begin{equation}\label{sln_brack}
\{\tau_\alpha,\tau_\beta\}= \sum_{p\in \alpha \cap \beta} \ve(p,\alpha,\beta) \left(\tau_{\alpha_p\beta_p}- \frac{\tau_{\alpha}\tau_{\beta}}{n}\right),
\end{equation}
where $\alpha,\beta$ are any smooth closed oriented loops in $F$ in general position.
(We identify closed oriented loops in $F$ with conjugacy classes in $\pi_1(F).$)
$\alpha \cap \beta$ is the set of the intersection points and
$\alpha_p\beta_p$ is the product of $\alpha$ and $\beta$ in $\pi_1(F,p),$
and $\ve(p,\alpha,\beta)$ is the sign of the intersection:

\centerline{\parbox{2in}{\psfig{figure=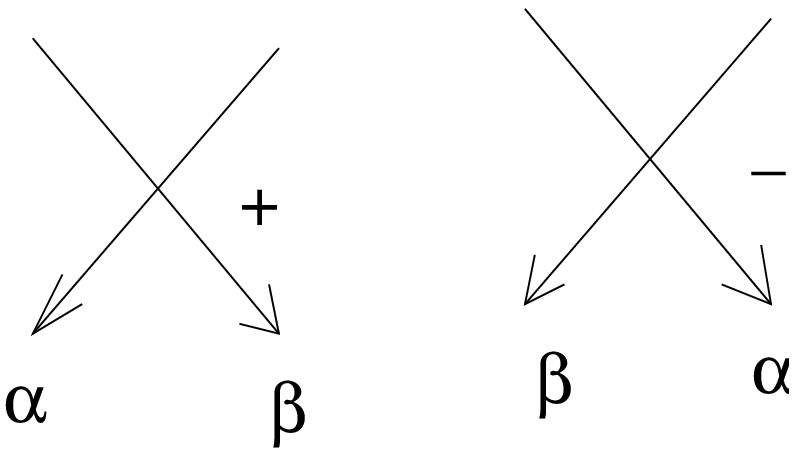,height=.7in}}}

(2) For $G=SO(n,\C), Sp(n,\C),$
\begin{equation}\label{spso_brack}
\{\tau_{\alpha},\tau_{\beta}\}= \sum_{p\in \alpha \cap \beta} \ve(p,\alpha,\beta)
\left(\tau_{\alpha_p\beta_p}-\tau_{\alpha_p\beta_p^{-1}}\right).
\end{equation}
\end{proposition}

\begin{proof}
Let $f_\alpha=Re\, \tau_\alpha$ and $\tilde f_\alpha=Im\, \tau_\alpha.$
The formulas for Goldman bracket between $f_\alpha$'s and $\tilde f_\alpha$'s
appear in \cite{Go2}. Since Goldman's $B$ and $f_\alpha$'s are twice
ours, the formulas for $SL(n,\C)$ are
\begin{equation}\label{goldman1}
\{f_\alpha,f_\beta\}= \frac{1}{2}\sum_{p \in \alpha\cap \beta} \ve(p, \alpha,\beta)\left(f_{\alpha_p\beta_p}-\frac{\tilde f_\alpha \tilde f_\beta-f_\alpha f_\beta}{n}\right),
\end{equation}
\begin{equation}\label{goldman2}
\{f_\alpha,\tilde f_\beta\}=\{\tilde f_\alpha,f_\beta\}=
\frac{1}{2}\sum_{p \in \alpha\cap \beta} \ve(p, \alpha,\beta)
\left(\tilde f_{\alpha_p\beta_p}-\frac{\tilde f_\alpha f_\beta- f_\alpha \tilde f_\beta}{n}\right),
\end{equation}
$$\{\tilde f_\alpha,\tilde f_\beta\}=-\{f_\alpha,f_\beta\}.$$
Now (\ref{sln_brack}) follows from
\begin{equation}\label{expansion}
\{\tau_\alpha,\tau_\beta\}=\{f_\alpha+i\tilde f_\alpha,f_\beta+i \tilde f_\beta\}=
2\{f_\alpha,f_\beta\}+2i\{\tilde f_\alpha,f_\beta\}.
\end{equation}

Let $G$ be $SO(n,\C)$ or $Sp(2n,\C)$ now.
By \cite{Go2},
$$\{f_\alpha,f_\beta\}=
\left(f_{\alpha_p\beta_p}-f_{\alpha_p\beta_p^{-1}}\right).$$
(Since $so(n,\C)=o(n,\C),$ the above formula holds for the Goldman bracket for $SO(n,\C)$ as well.)
Furthermore, from \cite[Lemma 1.11]{Go2} and Goldman's product formula, \cite{Go2},
we obtain additional formulas
$$\{\tilde f_\alpha,f_\beta\}= \{f_\alpha,\tilde f_\beta\}=\sum_{p \in \alpha\cap \beta}
\ve(p, \alpha,\beta) \left(\tilde f_{\alpha_p\beta_p}- \tilde f_{\alpha_p\beta_p^{-1}}\right).$$
$$\{\tilde f_\alpha,\tilde f_\beta\}=-\{f_\alpha,f_\beta\}.$$
Now (\ref{spso_brack}) follows from (\ref{expansion}).
\end{proof}

In a torus $T$, the signed number of the intersection points of any two curves $(a,b)$,
$(c,d)$ in $T$ is $\left|\begin{array}{cc} a & b \\ c & d\\ \end{array}\right|$ and, therefore,
formulas (\ref{sln_brack}) and (\ref{spso_brack}) for Goldman bracket on $\C[X_G(\Z^2)]$ simplify to

\begin{equation}\label{sln_tor_brack}
\{\tau_{a,b},\tau_{c,d}\}= \left|\begin{array}{cc} a & b \\ c & d\\ \end{array}\right|
\left(\tau_{a+c,b+d}-\frac{\tau_{a,b}\tau_{c,d}}{n}\right),\ \text{for\ } G=SL(n,\C),
\end{equation}
and
\begin{equation}\label{spso_tor_brack}
\{\tau_{a,b},\tau_{c,d}\}= \left|\begin{array}{cc} a & b \\ c & d\\ \end{array}\right|
\left(\tau_{a+c,b+d}-\tau_{a-c,b-d}\right)\ \text{for}\ G=SO(n,\C), Sp(2n,\C).
\end{equation}


\section{Proof of Theorem \ref{AgW_quant}}
\label{s_poisson}


In light of isomorphism (\ref{isoms}), it is enough to prove that Poisson bracket (\ref{Poisson})
on $\A_\g^W\otimes \C_0$, which we will denote here by $\{\cdot,\cdot\}_A,$ coincides with
Goldman bracket, (\ref{sln_tor_brack}) and (\ref{spso_tor_brack}), on $\C[X_G^0(\Z^2)]$ which we denote here by
$\{\ ,\ \}_G.$

Furthermore, since Poisson brackets are skew-commutative and they satisfy Leibniz' law:
$$\{f,gh\} = \{f,g\}h + g\{f,h\}$$
it is enough to prove that the above brackets coincide for algebra generators only.
Hence Theorem \ref{AgW_quant} follows from Proposition \ref{char_gen} and Propositions \ref{brack_sln} and \ref{brack_spso}.

\begin{proposition}\label{brack_sln}
For $G=SL(n,\C)$
$$\Theta\left(\{\tau_{a,b},\tau_{c,d}\}_A\right)=\{\Theta(\tau_{a,b}),\Theta(\tau_{c,d})\}_G,$$
for any $a,b,c,d\in \Z,$
where $\Theta$ is defined in (\ref{isoms}).
\end{proposition}

\begin{proof}
Let $\alpha_i:h\to \C,$ $i=1,...,n$ be the weights defined in Section \ref{ss_sln_unknot}.
Denote $E_{\alpha_i}$ and $Q_{\alpha_i}$ by $E_i$ and $Q_i$ respectively.
Then
\begin{equation}\label{theta_tau}
\Theta\left(\sum_{i=1}^n E_i^a Q_i^b\right)=\tau_{a,b}.
\end{equation}
Hence, by (\ref{Poisson}), $\{\Theta^{-1}\tau_{a,b}, \Theta^{-1}\tau_{c,d}\}_A$
is the image of
\begin{equation}\label{a_brack}
\frac{1}{h} \left(\sum_{i=1}^n E_i^aQ_i^b\cdot \sum_{i=1}^n E_i^cQ_i^d
-\sum_{i=1}^n E_i^cQ_i^d\cdot \sum_{i=1}^n E_i^aQ_i^b \right)
\end{equation}
in $\A_{sl(n,\C)}^W\otimes \C_0.$
By (\ref{qWeyl_rel}) and (\ref{killing_sln}),
$$\sum_{i=1}^n E_i^aQ_i^b\cdot \sum_{i=1}^n E_i^cQ_i^d=
\sum_{i,j} E_i^aE_j^cQ_i^bQ_j^d\cdot \begin{cases} \frac{1}{n}hbc & \text{for $i\ne j$}\\
(\frac{1}{n}-1)hbc & \text{for $i=j$} \end{cases}\quad \text{mod}\ h^2$$
for $q=e^h.$
Hence (\ref{a_brack}) equals
$$\sum_i \left(\frac{1}{n}-1\right)(bc-ad)E_i^{a+c}Q_i^{b+d}+\sum_{i\ne j}
\frac{1}{n}(bc-ad)E_i^aE_j^cQ_i^bQ_j^d=$$
$$\left|\begin{array}{cc} a & b \\ c & d\\ \end{array}\right|\cdot
\sum_i E_i^{a+c}Q_i^{b+d}- \frac{1}{n}\sum_{i,j} E_i^aE_j^cQ_i^bQ_j^d$$
and, by (\ref{theta_tau}), it equals to
$$\left|\begin{array}{cc} a & b \\ c & d\\ \end{array}\right|\cdot\left(
\Theta^{-1}\tau_{a+c,b+d}-\frac{\Theta^{-1} \tau_{a,b}\cdot \Theta^{-1}\tau_{a,b}}{n}\right).$$
Now the statement follows from (\ref{sln_tor_brack}).
\end{proof}

\begin{proposition}\label{brack_spso} For $G=SO(n,\C),Sp(2n,\C),$
$$\Theta\left(\{\tau_{a,b},\tau_{c,d}\}_A\right)=
\{\Theta(\tau_{a,b}),\Theta(\tau_{c,d})\}_G,$$
for any $a,b,c,d\in \Z.$
\end{proposition}

\begin{proof}

Let $s_n$ be the $n\times n$ matrix $$s_n=\left(\begin{array}{ccccc} 0 & ... & 0 & 1\\
0 & ... & 1 & 0\\
: & ... & : & :\\
1 & ... & 0 & 0\\ \end{array}\right).$$
Define matrices
$$B_m=\left(\begin{array}{cc} 0 & s_n\\ s_n & 0 \end{array}\right) y,\ \text{for}\
m=2n,\quad
B_m(x,y)=\left(\begin{array}{ccc} 0 & s_n & 0\\ s_n & 0 & 0\\
0 & 0 & 1\\ \end{array}\right),\ \text{for}\
m=2n+1,$$
$$\text{and}\quad \Omega_{2n}=\left(\begin{array}{cc} 0 & -s_n\\ s_n & 0 \end{array}\right).$$
Then $Sp(2n)\subset SL(m,\C)$ for $m=2n$ and $SO(m,\C)\subset SL(m,\C)$ are the groups of
isomorphisms of $\C^m$ preserving bilnear forms $(x,y)\to x^T\Omega_my$ and $(x,y)\to x^TB_my$, respectively.
The advantage of this definition of $SO(n,\C)$ over the "standard" one,
$SO(n,\C)=\{A: A\cdot A^T=I_n,\ det(A)=1\},$
 is that the intersection of the group of
diagonal matrices in $GL(n,\C)$ with $SO(n,\C)$ (defined above) is a maximal torus in $SO(n,\C).$
The Lie algebras $so(m)$ and $sp(m)$ for $m$ even are spaces of matrices $X$ such that
$B_mX+XB_m=0$ and $\Omega_m X+X\Omega_m=0$ respectively.
If $\g=sp(2n),so(2n),so(2n+1),$ then $H_i=E_{ii}-E_{n+i,n+i},$ for $i=1,...,n,$ form a basis of
its Cartan subalgebra $\h$.  Furthermore, the vectors $H_1,...,H_n$ are of equal length
and are mutually orthogonal with respect of the Killing form.

The weight lattice is generated by weights $\alpha_1,...,\alpha_n$ dual to $H_1,..,H_n,$
$\alpha_i(H_j)=\delta_{ij},$ and the dual Killing form on $h^*$ is given by
\begin{equation}\label{sp_kill}
(\alpha_i,\alpha_j)=\delta_{ij}.
\end{equation}

The Weyl group is composed all signed permutations $W=(\Z/2)^n \rtimes S_n $
for $sp(2n)$ and $so(2n+1)$ and it is the subgroup of $(\Z/2)^n \rtimes S_n $ composed of signed
permutations with an even number of sign changes, $(\Z/2)^{n-1}\rtimes S_n,$ for $so(2n).$

Let $E_i=E_{\alpha_i}$ and $Q_i=Q_{\alpha_i}$ for $i=1,...,n.$
We have
$$\Theta^{-1}(\tau_{a,b})=\sum_{i=1}^n (E_i^a Q_i^b+ E_i^{-a} Q_i^{-b}).$$
By (\ref{qWeyl_rel}) and (\ref{sp_kill}),
$$E_i^aQ_i^b\cdot E_i^cQ_i^d=
-hbc E_i^{a+c}Q_i^{b+d}\quad \text{mod}\ h^2,$$
for $q=e^h.$ Therefore, $\Theta^{-1}(\tau_{a,b})\cdot \Theta^{-1}(\tau_{c,d})=$
$$-hbc \sum_{i=1}^n E_i^{a+c}Q_i^{b+d} + E_i^{-a-c}Q_i^{-b-d}
-E_i^{a-c}Q_i^{b-d} -E_i^{-a+c}Q_i^{-b+d}\ \text{mod\ } h^2.$$
By (\ref{Poisson}),
$$\{\tau_{a,b},\tau_{c,d}\}_A=
\left|\begin{array}{cc} a & b \\ c & d\\ \end{array}\right|\cdot
\left(\tau_{a+c,b+d}-\tau_{a-c,b-d}\right).$$
Now the statement follows from  (\ref{spso_tor_brack}).
\end{proof}

\vspace*{.2in}\ \\
\centerline{Dept. of Mathematics, 244 Math. Bldg.}
\centerline{University at Buffalo, SUNY}
\centerline{Buffalo, NY 14260, USA}
\centerline{asikora@buffalo.edu}

\end{document}